\documentclass[11pt]{article}
\usepackage{amssymb}
\usepackage{amsthm}
\usepackage{amsmath,amsfonts,amssymb}
\usepackage{graphicx}
\usepackage{algpseudocode}
\usepackage{algorithm}
\usepackage{natbib}

\newtheorem{theorem}{Theorem}[section]
\newtheorem{lemma}[theorem]{Lemma}

\newtheorem{remark}[theorem]{Remark}

\def\sqr#1#2{\vbox{\hrule height .#2pt
\hbox{\vrule width .#2pt height #1pt \kern #1pt
\vrule width .#2pt}\hrule height .#2pt }}

\def\begi{\begin{itemize}}
\def\endi{\end{itemize}}
\def\bega{\begin{array}}
\def\enda{\end{array}}

\def\forall{\hbox{for every~ }}

\def\forall{\hbox{for all}~}

\def\bel{\begin{equation}\label}
\def\eeq{\end{equation}}

\begin{document}

\title{A Link-based Mixed Integer LP Approach for Adaptive Traffic Signal Control}
\author{Ke Han$^{a}\thanks{Corresponding author, e-mail: kxh323@psu.edu}$
\qquad Terry L. Friesz$^{b}\thanks{e-mail: tfriesz@psu.edu}$
\qquad Tao Yao$^{b}\thanks{e-mail: tyy1@engr.psu.edu}$ \\\\
%EndAName
$^{a}$\textit{Department of Mathematics}\\
\textit{Pennsylvania State University, University Park, PA 16802, USA}\\
$^{b}$\textit{Department of  Industrial and Manufacturing Engineering}\\
\textit{Pennsylvania State University, University Park, PA 16802, USA}}
\date{}
\maketitle

\begin{abstract}
This paper is concerned with adaptive signal control problems on a road network, using a link-based kinematic wave model \citep{HPFY}. Such a model employs the Lighthill-Whitham-Richards model with a triangular fundamental diagram. A variational type argument \citep{Lax, Newell} is applied so that the system dynamics can be determined without knowledge of the traffic state in the interior of each link. A Riemann problem for the signalized junction is explicitly solved; and an optimization problem is formulated in continuous-time with the aid of binary variables. A time-discretization turns the optimization problem into a {\it mixed integer linear program} (MILP). Unlike the cell-based approaches \citep{Daganzo 1995, Lin and Wang, Lo 1999a},  the proposed framework does not require modeling or computation within a link, thus reducing the number of (binary) variables and computational effort.

The proposed model is free of vehicle-holding problems, and captures important features of signalized networks such as physical queue, spill back, vehicle turning, time-varying flow patterns and dynamic signal timing plans. The MILP can be efficiently solved with standard optimization software.
\end{abstract}

\section{\label{Intro}Introduction}

Traffic signal is an essential element to the management of the transportation network. For the past several decades, signal control strategies have evolved from ones developed based on historical information, often referred to as the fixed timing plan, to the generation of control strategies in which the control system is fully responsive. In the latter case, the cycle lengths and splits of the signal are determined based on real-time information. Representatives of such signal-control systems are OPAC \citep{Gartner}, RHODES \citep{MH}, SCAT \citep{SD} and SCOOT \citep{HRBR}.

The performance of a traffic signal control system depends on the optimization procedure embedded therein. We distinguish between two optimization procedures: 1) heuristic approach, such as those developed with feedback control, genetic algorithms and fuzzy logic; and 2) exact approach, such as those arising from mathematical control theory and  mathematical programming. Among these exact approaches, the {\it mixed integer programs} (MIPs) are of particular interest and has been used extensively in the signal control literature.  \cite{Improta and Cantarella} formulated the traffic signal control problem for a single road junction as a mixed binary integer program. \cite{Lo 1999a} and \cite{Lo 1999b} employed the {\it cell transmission model} (CTM) \citep{Daganzo 1994, Daganzo 1995} and casted a signal control problem as mixed integer linear program. In these papers, the author addressed time-varying traffic patterns and dynamic timing plan.  In \cite{Lin and Wang}, the same formulation based on CTM was applied to capture more realistic features of signalized junctions such as the total number of vehicle stops and signal preemption in the presence of emergency vehicles. One subtle issue associated with  CTM-based mathematical programs is the phenomena known as traffic holding, which stem from the linear relaxation of the nonlinear dynamic.  Such an action induces the unintended holding of vehicles, i.e., a vehicle is held at a cell even though there is capacity available downstream for the vehicle to advance. The traffic holding can be avoided by introducing additional binary variables, see \cite{Lo 1999}. However, this approach ends up with a substantial amount of binary variables and yields the program computationally demanding. An alternative way to treat holding problem is to manipulate the objective function such that the optimization mechanism enforces the full utilization of available capacities in the network. This approach however, strongly depends on specific structure of the problem and the underlying optimization procedure. Specific discussion on traffic holding can be found in \cite{SNZ}

This paper presents a novel MILP formulation for signal control problem based on the {\it link-based kinematic wave model} (LKWM). This model was proposed in \cite{HPFY} as a continuous-time extension of the Lighthill-Whitham-Richards (LWR) model \citep{Lighthill and Whitham, Richards} to networks. This model describes network dynamics with variables associated to the entrance and exit of each link. It employs a Newell-type variational argument \citep{Newell, VT1} to capture shock waves and vehicle spillback. Analytical properties of this model pertaining to solution existence, uniqueness and well-posedness are provided in \cite{HPFY}. A discrete-time version of the LKWM, known as the {\it link transmission model}, was discussed in \cite{LTM}. In contrast to the cell-based math programming approaches where the variables of interest correspond to each cell and each time interval, the model proposed in this paper is link-based, i.e.  the variables are associated with each link and each time interval. The resulting MILP thus substantially reduces the number of (binary) variables and hence the computational effort. In addition, the link-based approach prevents vehicle holding within a link without using binary variables. The resulting mixed integer linear program can be solved efficiently with commercial optimizers such as CPLEX.

The formulation in this paper captures key phenomena of vehicular flow at junctions such as the formation, propagation and dissipation of physical queues, spill back and vehicle turning. It also considers important features of signal control such as dynamic timing plan and time-varying flow patterns.  In the remainder of this introduction, we briefly review the LWR model and the variational method for solving the corresponding Hamilton-Jacobi equation. This will serve as preliminary background as we proceed in Section \ref{LLWR} to discuss our link-based network mode.

\subsection{\label{LWR} Lighthill-Whitham-Richards model}

Following the classical model introduced by \cite{Lighthill and Whitham} and \cite{Richards}, we model the traffic dynamics on a link with the following first order  {\it partial differential equation} (PDE), which describes the spatial-temporal evolution of density and flow
\begin{equation}\label{LWRPDE}
{\partial\over \partial t}\rho(t,\,x)+{\partial\over\partial x} f\big(\rho(t,\,x)\big)~=~0
\end{equation}
where $\rho(t,\,x): [0,\,+\infty)\times[a,\,b]\rightarrow [0,\,\rho_{j}]$ is average vehicle density, $f(\rho): [0,\,\rho^{jam}]\rightarrow [0,\,C]$ is average flow. $\rho^{jam}$ is jam density, $C$ is flow capacity.  The function $f(\cdot)$ articulates a density-flow relation and is commonly referred to as the fundamental diagram.

Classical mathematical results on the first-order hyperbolic equations of the form (\ref{LWRPDE}) can be found in \cite{Bressan 2000}. For a detailed discussion of numerical schemes for conservation laws, we refer the reader to \cite{Godunov} and \cite{LeVeque}. A well-known discrete version of the LWR model, the Cell Transmission Model (CTM), was introduced by \cite{Daganzo 1994, Daganzo 1995}. PDE-based models have been studied extensively also in the context of vehicular networks, with a list of selected references including \cite{Bretti et al, Coclite et al, Daganzo 1995, Herty and Klar, Holden and Risebro, Jin 2010, Jin and Zhang 2003, Lebacque and Khoshyaran 1999, Lebacque and Khoshyaran 2002}.

\subsection{The Hamilton-Jacobi equation and the Lax-Hopf formula} 

Initially introduced by \cite{Lax, Lax1}, then extended in \cite{ABP} and \cite{LeFloch}, and applied to traffic theory in \cite{CC1, VT1, LWRDUE, HPFY}, the Lax-Hopf formula provides a new characterization of the solution to the hyperbolic conservation law and Hamilton-Jacobi equation. The Lax formula is derived from the characteristics equations associated with the H-J equation, which arises in the classical calculus of variations and mathematical mechanics. The reader is referred to \cite{Evans} for a detailed discussion.

Let us introduce the function $N(\cdot,\,\cdot): [0,\,+\infty)\times[a,\,b]\rightarrow \mathbb{R}$, such that
\begin{equation}\label{Ndef}
{\partial\over \partial x} N(t,\,x)~=~-\rho(t,\,x),\qquad{\partial\over \partial t} N(t,\,x)~=~f\big(\rho(t,\,x)\big)
\end{equation}
 The function $N(t,\,x)$ is sometimes referred to as Moskowitz function or Newell-curves. It has been studied extensively, in \cite{CC1, VT1, Moskowitz, Newell}. A well-known property of $N(\cdot,\,\cdot)$ is that it satisfies the following Hamilton-Jacobi equation
\begin{equation}\label{HJeqn}
{\partial\over \partial t} N(t,\,x)-f\left( -{\partial\over\partial x} N(t,\,x)\right)~=~0\qquad (t,\,x)\in[0,\,+\infty)\times [a,\,b]
\end{equation}

\noindent Let $\Omega$ be a subset of  $[0,\,+\infty)\times[a,\,b]$, and $c(\cdot,\,\cdot): \Omega\rightarrow\mathbb{R}_+$ be a value condition which prescribes the value of $N(\cdot,\,\cdot)$ on $\Omega$. The solution to the H-J equation (\ref{HJeqn}) with condition $c(\cdot,\,\cdot)$ is given by the Lax-Hopf formula \cite{CC1, VT1, LeFloch}. 

\begin{equation}\label{LHformula}
N(t,\,x)~=~\inf\big\{c(t-T,\,x-T\,u)+T\,f^*(u)\big\} 
\end{equation}
such that $u\in[f'\big(\rho^{jam}-\big),\,f'(0+)]$, $T\geq 0$ and $(t-T,\,x-T\,u)\in \Omega$. Where $f^*(\cdot)$ is the Legendre (concave) transformation of $f(\cdot)$.

\subsection{Oganization}
The rest of this article is organized as follow. In section \ref{LLWR}, we present a network  model known as the link-based kinematic wave model \citep{HPFY}. Section \ref{secTS} formulates the traffic signal control problem in both continuous- and discrete-time, based on the LKWM. The discrete-time problem is further formulated as a mixed integer linear program. Section \ref{secNE} presents a numerical example, which demonstrates and evaluates the proposed formulation.

\section{\label{LLWR} Link-based Kinematic Wave Model}
In this section, we present a kinematic wave model on networks, with a triangular fundamental diagram for each link. Unlike the cell-based models \cite{Daganzo 1994, Daganzo 1995}, the proposed model does not require modeling or computation in the interior of the link. For this reason, we call it the link-based kinematic wave model.

\subsection{\label{statevariables} State variables of the system}

Consider the link represented by an interval $[a,\,b]$, with $b-a=L>0$. In the derivation of the LKWM, we select flow $q(t,\,x)$ and regime $r(t,\,x)$ as the state variables for the link, instead of density.  It is obvious  that a single value of $q$ corresponds to two traffic states: 1) the free flow phase ($r=0$); and 2) the congested phase ($r=1$). Therefore, the pair $(q(t,\,x),\,r(t,\,x))\in [0,\,C]\times \{0,\,1\}$ determines a unique density value. This simple observation gives rise to the following map
\begin{equation}\label{psidef}
\psi(\cdot): [0,\,C]\times \{0,\,1\}~\rightarrow ~ [0,\,\rho^{jam}],\qquad ( q,\,r)\mapsto \rho
\end{equation}

%
%\begin{figure}[h]
%\centering
%\includegraphics[width=0.4\textwidth]{fd1.eps}
%\caption{Illustration of the new state variable $(q,\,r)$. }
%\label{figfd1}
%\end{figure}
%
%
%

\subsection{Riemann problem at a junction with one incoming link}\label{secjunction}

Extension of the kinematic wave model to a network turns out to be subtle; the issues associated therein are 1) a proper definition of a weak entropy solution at a junction of arbitrary topology; 2) uniqueness and well-posedness of the entropy solution. The reader is referred to \cite{HPFY, Garavello and Piccoli, Jin 2010, Jin and Zhang 2003} for some specific discussion. A junction model can be analyzed by considering a Riemann problem, which is an initial value problem with constant datum on each incoming and outgoing link. Due to space limitation, instead of a comprehensive discussion on various types of Riemann problems, we focus on the Riemann problem for a particular junction, that is, the one with one incoming link and $n\geq 2$ outgoing links. This is because we assume that during one signal phase, cars from only one incoming link can enter the junction.

In order to model vehicle turning, we fix a traffic distribution matrix 
$$
A~=~\begin{pmatrix} \alpha_{1,2}  & \alpha_{1,3} & \ldots  & \alpha_{1,n+1}\end{pmatrix}
$$
where $0\leq \alpha_{1, i}\leq 1,\, i=2,\ldots, n+1$, $\sum_{i=2}^{n+1}\alpha_{1,i}=1$. The coefficients $\alpha_{1,i}$ determines how the traffic from the incoming link $I_1$ distributes in percentages to the outgoing link $I_i$. For simplicity, we assume $A$ is time-independent. Note that there is no substantial difficulty with transforming our modeling framework to deal with time-varying distribution matrices; such extension, however, requires additional information on route choices, which is beyond the scope of this paper.

The next theorem characterizes the solution to the Riemann problem  at junction with one incoming link. 

\begin{theorem}\label{thm1}
Consider a junction with one incoming link $I_1$ and $n\geq 2$ outgoing links $I_{2},\ldots, I_{n+1}$. For every initial data $y_{1,0}, \ldots, y_{n+1,0}\in [0,\,C]\times \{0,\,1\}$, there exists a unique $n+1$-tuple
$$
\hat y_1,\,\ldots, \, \hat y_{n+1}\in [0,\,C]\times \{0,\,1\}
$$
where $\hat y_i=(\hat q_i,\,\hat r_i)$, such that the solutions to the initial-boundary value problems at the junction
$$
\begin{cases}
{\partial\over \partial t}\rho(t,\,x)+{\partial\over\partial x} f_1\big(\rho(t,\,x)\big)~=~0\\
\rho(0,\,x)~= ~ \psi\big(y_{1,0}\big) \\
\rho(t,\,b_1)~=~\psi\big(\hat y_1\big)
\end{cases}
\qquad 
\begin{cases}
{\partial\over \partial t}\rho(t,\,x)+{\partial\over\partial x} f_i\big(\rho(t,\,x)\big)~=~0\\
\rho(0,\,x)~= ~ \psi\big(y_{i,0}\big) \\
\rho(t,\,a_i)~=~\psi\big(\hat y_i\big)
\end{cases}
\quad j~=~2,\,\ldots,\,n+1
$$
is the admissible weak solution to the junction problem in the sense defined in \cite{Coclite et al}. In addition, we have the following characterization: the boundary states $(\hat q_i,\,\hat r_i),\, i=1,\ldots, n+1$ are given by

\begin{equation}\label{thm1eqn3}
\hat q_1~=~\min\left\{q_1^{max},\,{q_2^{max}\over \alpha_{1,2}},\, {q_3^{max}\over \alpha_{1,3}},\,\ldots,\, {q_{n+1}^{max}\over \alpha_{1, n+1}}\right\}
\end{equation}
\begin{equation}\label{thm1eqn4}
\hat r_1~=~\begin{cases}
1,\qquad &\hbox{if}\quad r_{1,0}~=~1\\
0, \qquad &\hbox{if}\quad r_{1,0}~=~0,~ \hat q_1~=~q_{1,0}\\
1, \qquad &\hbox{if}\quad r_{1,0}~=~0,~ \hat q_1~<~q_{1,0}
\end{cases}
\end{equation}
\begin{equation}\label{thm1eqn5}
\hat q_i~=~\alpha_{1,i}\,\hat q_1,\qquad i~=~2,\,\ldots,\,n+1
\end{equation}
\begin{equation}\label{thm1eqn6}
\hat r_i~=~\begin{cases}
0,\qquad & \hbox{if}\quad r_{i,0}~=~0\\
1, \qquad &\hbox{if}\quad r_{i,0}~=~1,~\hat q_i~=~q_{i,0}\\
0,\qquad &\hbox{if}\quad r_{i,0}~=~1,~\hat q_i~<~q_{i,0}
\end{cases}
\end{equation}
where
\begin{equation}\label{thm1eqn7}
 q_i^{max}~\doteq~\begin{cases}
 q_{i,0}+r_{i,0}(C_i-q_{i,0}),\qquad &i~=~1\\
 C_i+r_{i,0}(q_{i,0}-C_i),\qquad & i~=~2,\,\ldots,\,n+1
 \end{cases}
\end{equation}
\end{theorem}
\begin{remark}
The quantity $q_i^{max}$ is the maximum flux an incoming (outgoing) link can send (receive) --  a quantity identified as demand (supply) by \cite{Lebacque and Khoshyaran 1999, Lebacque and Khoshyaran 2002}.
\end{remark}

The aforementioned map $(y_{i,0})_{i=1,\ldots, n+1}\mapsto (\hat y_{i})_{i=1,\ldots, n+1}$ is commonly referred to as the Riemann solver, see \cite{Garavello and Piccoli} for a formal discussion. Theorem \ref{thm1} describes the Riemann solver using the new state variables $q$ and $r$. The verification of theorem \ref{thm1} is straightforward. For junctions with arbitrary topology, the Riemann Solvers are not available in closed-form. Yet, in the case of one incoming link, we are able to express the Riemann solver explicitly.

In the next subsection, we analyze shock formation and propagation within one single link. The location of the shock wave is crucial as it determines the regime variable $r$ associated with the two boundaries of the link.

\subsection{\label{shock} Shock formation and propagation within the link}

We focus on solutions generated by assuming an initially empty network, i.e. $y_{i,0}=(0,\,0)$. The key to our analysis is the location of a so-called {\it separating shock}, which divides each link into two zones: free flow zone $(r=0)$, and congested zone $(r=1)$.  We begin with the fact that if the network is initially empty, then there can be at most one separating shock on each link

\begin{lemma}\label{sshocklemma}
For every link $I_i$ and any solution $y_i(t,\,x)=\big(q_i(t,\,x),\,r_i(t,\,x)\big)$ with $y_i(0,\,x)=(0,\,0)$, the following statement holds: 
\begin{enumerate}
\item For every $t\geq 0$, there exists at most one $x_i^*(t)\in(a_i,\,b_i)$ such that $r_i(t,\,x^*(t)-)<r_i(t,\,x^*(t)+)$
\item For all $x\in[a_i,\,b_i]$, 
\begin{align*}
&r_i(t,\, x)~=~0,\qquad \hbox{if}~~ x~<~x_i^*(t)\\
&r_i(t,\,x)~=~1,\qquad \hbox{if}~~ x~>~x_i^*(t)
\end{align*}
\end{enumerate}
\end{lemma}
\begin{proof}
See \cite{Bretti et al}.
\end{proof}

According to Lemma \ref{sshocklemma}, the separating shock  emerges from the downstream boundary $b_i$ of the link, and propagates towards the interior of the link. The speed of this separating shock is given by the  Rankine-Hugoniot condition \cite{Evans}.
%
%\begin{equation}\label{RH}
%{d\over dt} x_i^*(t)~=~{q_i\big(t,\, x_i^*(t)-\big)-q_i\big(t,\,x_i^*(t)+\big)\over \rho_i\big(t,\,x_i^*(t)-\big)-\rho_i\big(t,\,x_i^*(t)+\big)},\qquad \hbox{almost every}~t
%\end{equation}
%The shock speed is interpreted as the slope of the line segment connecting the two states  on the fundamental diagram. See Figure \ref{figsshock}.
%
%
%\begin{figure}[h!]
%\centering
%\includegraphics[width=0.8\textwidth]{sshock.eps}
%\caption{Example of the separating shock. Left figure: the $t - x$ domain is separated by the shock into two phases. Right figure: the speed of the separating shock ${d\over dt}x_i^*(t)$ is given by the Rankine-Hugoniot condition (\ref{RH}).}
%\label{figsshock}
%\end{figure}
%
%
%\begin{figure}[h!]
%\begin{minipage}[b]{.49\textwidth}
%\centering
%\includegraphics[width=\textwidth]{sshock.eps}
%\caption{\small Example of the separating shock. The $x$-axis represent spatial domain; the $t$-axis represent temporal domain.}
%\label{fig1}
%\end{minipage}
%\begin{minipage}[b]{.49\textwidth}
%\centering
%\includegraphics[width=\textwidth]{Light5.eps}
%\caption{\small Link exit flow.}
%\label{fig2}
%\end{minipage}
%\end{figure}
It is clear that as long as the separating shock remains in the interior of the link $I_i$,  the upstream and downstream boundary conditions  do not interact. Thus the exit of a link remains in the congested phase; while the entrance remains in the free flow phase. Consequently, the Riemann Solver  in Theorem \ref{thm1} is expressed entirely with exogenous parameters  $C_i$ and $\alpha_{1,i}$.  On the other hand, if the separating shock reaches either boundary, it becomes a latent shock. Two cases may arise.

\vspace{0.1 in}
\noindent i) The the shock reaches the exit, i.e. $x_i^*(t)=b_i$. In this case, the current link is dominated by free flow phase. The boundary condition at $x=a_i$ directly influences the boundary condition at $x=b_i$, in a way  expressed by
\begin{equation}\label{caseii}
q_i\left(t,\,b_i-\right)~=~q_i\left(t-{L_i\over k_i},\, a_i\right)
\end{equation}
where $L_i$ is the length of the link, $k_i$ is the speed of forward wave propagation. See Figure \ref{figextreme} for an illustration. 

\vspace{0.1 in}
\noindent ii) The shock reaches the entrance, i.e. $x_i^*(t)=a_i$. In this case,  the current link is dominated by congested phase. The  boundary condition at $x=b_i$ directly affects the boundary condition at $x=a_i$
\begin{equation}\label{casei}
q_i\left(t,\,a_i+\right)~=~q_i\left(t-{L_i\over w_i},\, b_i\right)
\end{equation}
where $w_i$ is the speed of backward wave propagation.

\begin{figure}[h!]
\centering
\includegraphics[width=.75\textwidth]{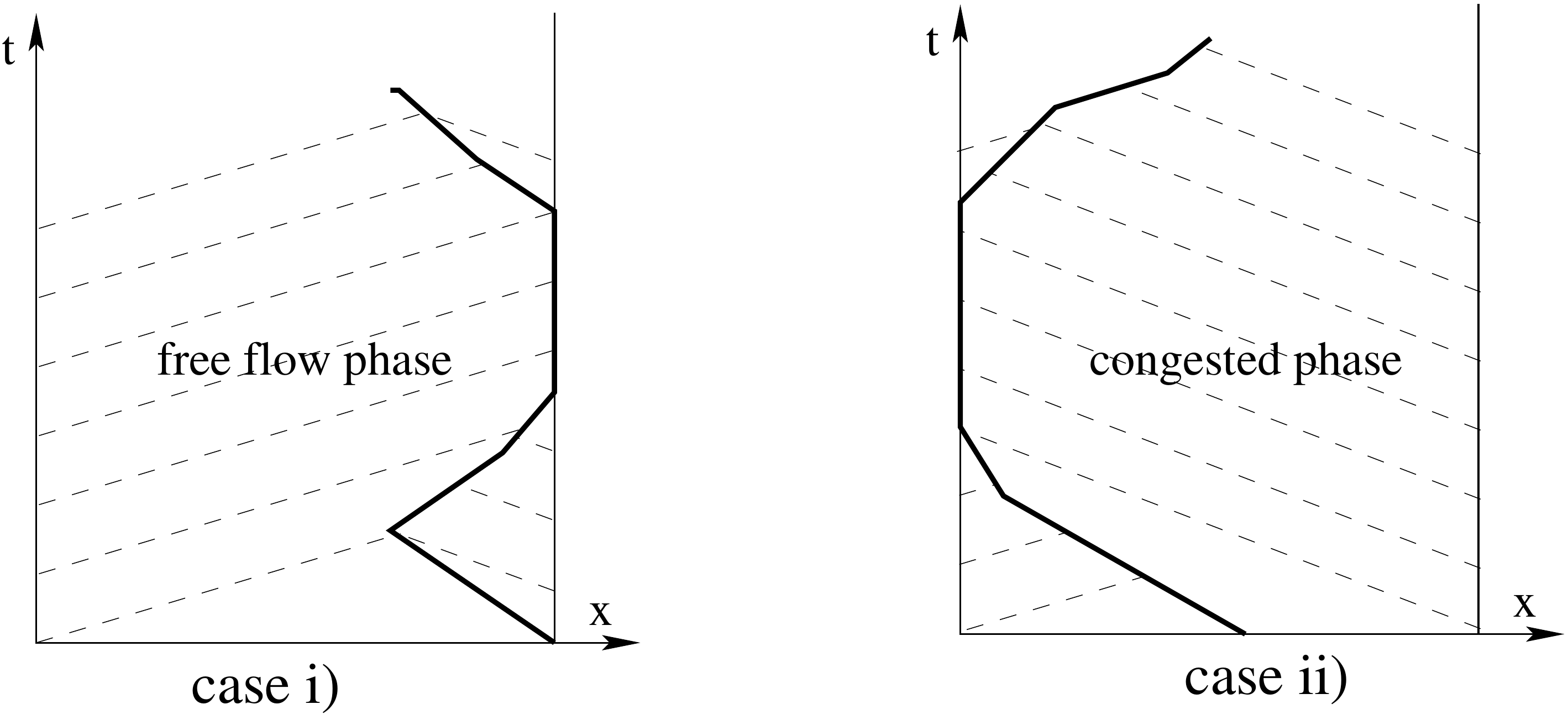}
\caption{Example of latent separating shocks. Case i): the separating shock reaches the right (downstream) boundary. Case ii): the separating shock reaches the left (upstream) boundary.}
\label{figextreme}
\end{figure}

In either case, the Riemann solver involves boundary flows (\ref{caseii}), (\ref{casei}), which are endogenous. The next key step towards the link-based flow model is the detection of these two extreme cases. This can be done with the Lax-Hopf formula.

\subsection{The variational approach for detecting latent shock}

This section provides sufficient and necessary condition for the occurrence of the latent shock. The derivation is omitted for brevity, we refer the reader  to \cite{HPFY} for a detailed discussion.  Define for each link $I_i$ the cumulative entering and exiting vehicle numbers 
$$
N_{i, up}(t)~\doteq~\int_0^t q_i(s,\,a_i)\,ds,\qquad N_{i, down}~\doteq~\int_0^t q_i(s,\,b_i)\,ds
$$

Recall that the separating shock $x_i^*(\cdot): [0,\,+\infty)\rightarrow [a_i,\,b_i]$ is a continuous curve in the $t - x$ domain. The following theorem provides sufficient and necessary condition for the occurrence of the latent shock. 
\begin{theorem}\label{thmdetection}
Let $N_i(\cdot,\,\cdot): [0,\,+\infty)\times[a_i,\,b_i]$ be the unique viscosity solution to the Hamilton-Jacobi equation (\ref{HJeqn}), satisfying zero initial condition, upstream boundary condition $N_{i, up}(\cdot)$ and downstream boundary condition $N_{i, down}(\cdot)$. Then for all $t\geq 0$,
\begin{align}
\label{latent3}
x_i^*(t)&~=~a_i~\Longleftrightarrow~N_{i, up}\left(t\right)~\geq~N_{i, down}\left(t-{L_i\over w_i}\right)+\rho_i^{jam}L_i\\
\label{latent4}
x_i^*(t)&~=~b_i~\Longleftrightarrow~ N_{i,up}\left(t-{L_i\over k_i}\right)~\leq~N_{i, down}\left(t\right)
\end{align}
\end{theorem}

\begin{remark}
The significance of criteria (\ref{latent3})-(\ref{latent4}) is that the two extreme cases can be detected without any computation within the link. This is because $N_{up}(\cdot)$, $N_{down}(\cdot)$ are determined completely by the boundary flows. 
Theorem \ref{thmdetection} is the key ingredient  of the LKWM, which allows the network model to be solved at the link level.

Analytical properties of the LKWM pertaining to solution existence, uniqueness and well-posedness are provided in \cite{HPFY}. 
\end{remark}

\section{Traffic signal control problem based on the LKWM}\label{secTS}

In this section,  the signal control problem is formulated with LKWM.   We start with a single junction, with two incoming links $I_1$, $I_2$, and two outgoing links $I_3$, $I_4$ (Figure \ref{figintersection}).
\begin{figure}[h!]
\centering
\includegraphics[width=.25\textwidth]{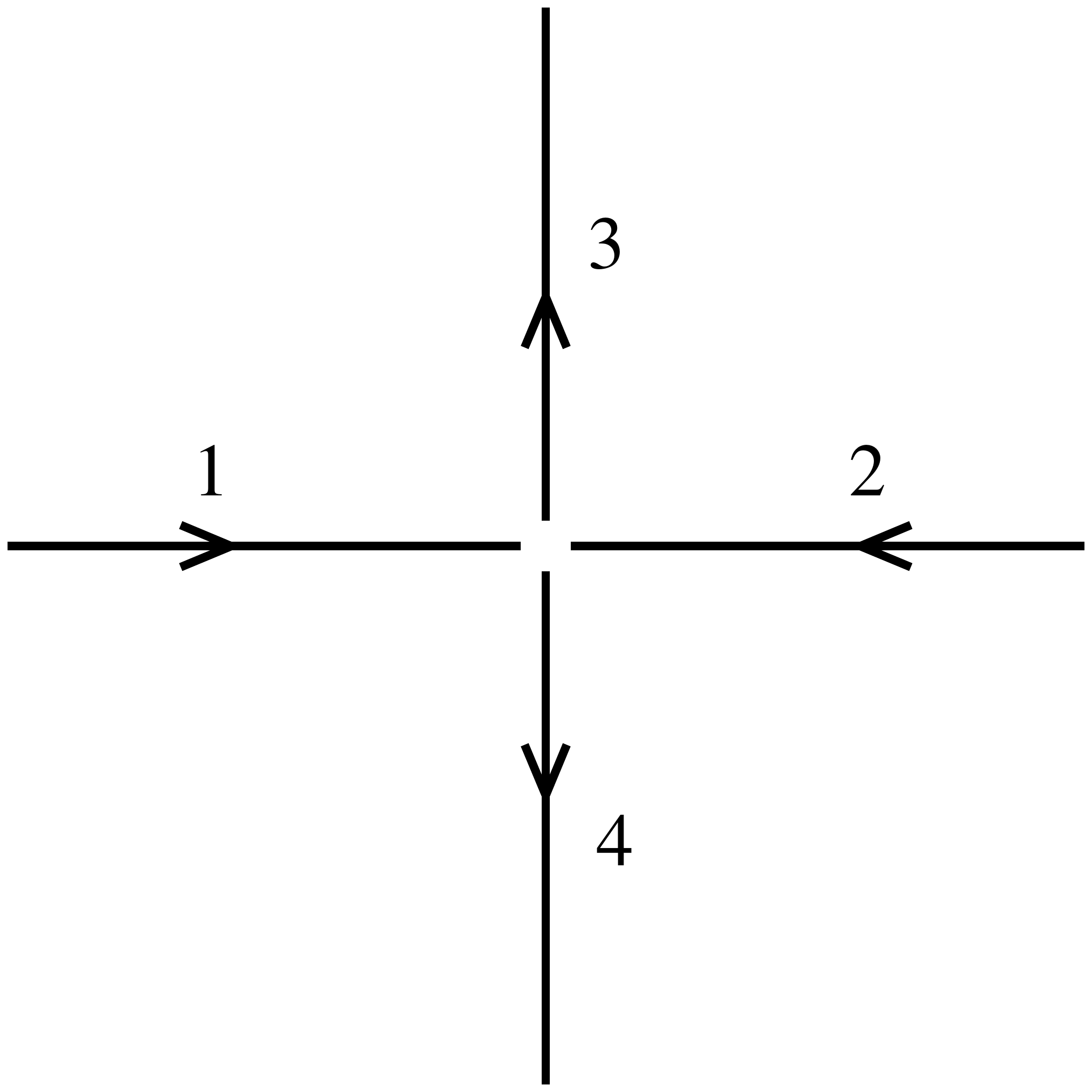}
\caption{A signalized junction with two incoming links and two outgoing links.}
\label{figintersection}
\end{figure}
Each link is represented by a spatial interval $[a_i,\,b_i],\, i=1,\,2,\,3,\,4$. The fundamental diagrams for each link is given by 
$$
f_i(\rho)~=~\begin{cases}
k_i\,\rho\qquad & \rho\in[0,\,\rho_i^*]\\
-w_i\,(\rho-\rho_i^{jam})\qquad & \rho\in(\rho_i^*,\,\rho_i^{jam}]
\end{cases},\qquad i~=~1,\,2,\,3,\,4
$$
with $C_i\doteq k_i\,\rho_i^*$ being the flow capacity. We make the following assumptions

\noindent {\bf A1} The network is initially empty. 

\noindent {\bf A2} Drivers arriving at the junction distribute on the outgoing roads according to some known coefficients:
$$
A~=~\begin{pmatrix} \alpha_{1,3}   &  \alpha_{1,4} \\ \alpha_{2,3}  &  \alpha_{2,4}\end{pmatrix}
$$ 
where $\alpha_{ij}$ denotes the percentage of traffic coming from link $I_i$ that distributes to outgoing link $I_j$.

In the problem setting, the flows $\bar q_i(\cdot),\,i=1,\,2$ entering links $I_1$ and $I_2$ are known. In practice, $\bar q_i(\cdot)$ can be measured at the entrance using fixed sensors such as loop-detectors.

\subsection{Continuous-time formulation}

In this section, we formulate the constraints  of the system in continuous time. In Section \ref{secdiscretetime} we will reformulate the system dynamic as linear constraints in discrete time using binary variables. Notice that the discussion in this section can be the building block for extension to networks with multiple intersections.

 Let us fix the planning horizon $[0,\,T]$ for some fixed $T>0$.  Introducing the piecewise-constant control variables $u_i(\cdot): [0,\,T]\rightarrow \{0,\,1\},\,i=1,\,2$, with the agreement that $u_i(t)=0$ if the light is red for link $I_i$, and $u_i(t)=1$ if the light is green for link $I_i$.  It is convenient to use the following set of notations. For $i=1,\,2,\,3,\,4$.
\begin{align*}
\bar q_i(\cdot), \qquad\qquad &\hbox{the flow of cars  entering link }I_i, \\
\hat q_i(\cdot), \qquad\qquad &\hbox{the flow of cars  exiting link }I_i, \\
\bar r_{i}(\cdot), \qquad  \qquad & \hbox{the binary variable that indicates the regime at } x~=~a_i+,\\
\hat r_{i}(\cdot), \qquad  \qquad & \hbox{the binary variable that indicates the regime at } x~=~b_i-,\\
\bar q_{i}^{max}(\cdot), \qquad\qquad &\hbox{the maximum flow allowed to enter the link }I_i, \\
\hat q_{i}^{max}(\cdot), \qquad\qquad &\hbox{the maximum flow allowed to exit the link }I_i,\\
N_{up,i}(\cdot), \qquad\qquad  & \hbox{the cumulative number of cars that have entered  link } I_i,\\
N_{down,i}(\cdot),\qquad\qquad & \hbox{the cumulative number of cars that have exited link } I_i,\\
u_i(\cdot) ,\qquad\qquad &\hbox{the signal control variable for link } I_i,
\end{align*}

\begin{theorem}\label{cpthm} The dynamics at the junction (Figure \ref{figintersection}) with signal control can be described by the following system of differential algebraic equations (DAE) with binary variables.
%\begin{equation}\label{cpeqn1}
%{d\over dt}N_{up,i}(t)~=~\bar q_i(t),\qquad {d\over dt}N_{down, i}(t)~=~\hat q_i(t),\qquad i~=~1,\,2,\,3,\,4
%\end{equation}
%\begin{equation}\label{cpeqn2}
%\bar r_i(t)~=~\begin{cases}
% 1,\qquad &\hbox{if}\quad N_{up,i}(t)~\geq~N_{down,i}\left(t-{L_i\over w_i} \right)+\rho_i^{jam}L_i\\
%0, \qquad  &\hbox{otherwise}
%\end{cases},\qquad i~=~1,\,2,\,3,\,4
%\end{equation}
%\begin{equation}\label{cpeqn3}
%\hat r_i(t)~=~\begin{cases}
% 0,\qquad &\hbox{if}\quad N_{up,i}\left(t-{L_i\over k_i}\right)~\leq~N_{down,i}(t)\\
%1, \qquad  &\hbox{otherwise}
%\end{cases},\qquad i~=~1,\,2,\,3,\,4
%\end{equation}
%\begin{equation}\label{cpeqn4}
%\bar q_i^{max}(t)~=~C_i+\bar r_i(t)\,\left( \hat q_i\left(t-{L_i\over w_i}\right)-C_i \right),\qquad i~=~1,\,2
%\end{equation}
%\begin{equation}\label{cpeqn5}
%\hat q_i^{max}(t)~=~\bar q_i\left(t-{L_i\over k_i}\right)+r_i(t)\,\left( C_i-\bar q_i\left(t-{L_i\over k_i}\right) \right),\qquad i~=~3,\,4
%\end{equation}
%\begin{equation}\label{cpeqn6}
%\hat q_i(t)~=~\begin{cases}
%0,\qquad &\hbox{if} \quad u_i(t)~=~0\\
% \min\left\{\hat q^{max}_i(t),\,{\bar q^{max}_3(t)\over\alpha_{i,3}},\,{\bar q_4^{max}(t)\over\alpha_{i,4}}\right\} \qquad &\hbox{if}\quad u_i(t)~=~1 
%\end{cases},\qquad i~=~1,\,2
%\end{equation}
%\begin{equation}\label{cpeqn7}
%\bar q_k(t)~=~\alpha_{1,k}\,\hat q_1(t)+\alpha_{2,k}\,\hat q_2(t),\qquad k~=~3,\,4,
%\end{equation}
%\begin{equation}\label{cpeqn8}
%u_1(t)+u_2(t)~=~1\qquad \forall ~t\in[0,\,T]
%\end{equation}
\begin{align}
\label{cpeqn1}
&{d\over dt}N_{up,i}(t)=\bar q_i(t),\qquad {d\over dt}N_{down, i}(t)=\hat q_i(t),\qquad \qquad\quad~ i=1,\,2,\,3,\,4\\
\label{cpeqn2}
\bar r_i(t)&=\begin{cases}
 1,\quad &\hbox{if}\quad N_{up,i}(t)\geq N_{down,i}\left(t-{L_i\over w_i} \right)+\rho_i^{jam}L_i\\
0, \quad  &\hbox{otherwise}
\end{cases},\quad~ i=1,\,2,\,3,\,4\\
\label{cpeqn3}
\hat r_i(t)&=\begin{cases}
 0,\quad &\hbox{if}\quad N_{up,i}\left(t-{L_i\over k_i}\right) \leq N_{down,i}(t)\\
1, \quad  &\hbox{otherwise}
\end{cases},\qquad\qquad\quad~\, i=1,\,2,\,3,\,4\\
\label{cpeqn4}
\bar q_i^{max}(t)&=C_i+\bar r_i(t)\left( \hat q_i\left(t-{L_i\over w_i}\right)-C_i \right),\qquad\qquad\qquad\qquad\quad~ i=1,\,2\\
\label{cpeqn5}
\hat q_i^{max}(t)&=\bar q_i\left(t-{L_i\over k_i}\right)+\hat r_i(t)\left( C_i-\bar q_i\left(t-{L_i\over k_i}\right) \right),\qquad\qquad\quad~ i=3,\,4\\
\label{cpeqn6}
\hat q_i(t)&=\begin{cases}
0,\quad &\hbox{if} \quad u_i(t)=0\\
 \min\left\{\hat q^{max}_i(t),\,{\bar q^{max}_3(t)\over\alpha_{i,3}},\,{\bar q_4^{max}(t)\over\alpha_{i,4}}\right\} \quad &\hbox{if}\quad u_i(t)=1 
\end{cases},\quad i=1,\,2\\
\label{cpeqn7}
&\qquad\qquad\qquad\bar q_k(t)=\alpha_{1,k}\,\hat q_1(t)+\alpha_{2,k}\,\hat q_2(t),\qquad\qquad\quad~ k=3,\,4,\\
\label{cpeqn8}
&\qquad\qquad\qquad\qquad u_1(t)+u_2(t)=1\qquad \forall ~t\in[0,\,T]
\end{align}
\end{theorem}
\begin{proof}
(\ref{cpeqn1}) is by definition. For $i=1,2,3,4$, if the separating shock on link $I_i$ reaches the entrance  $a_i$ (exit $b_i$), then the regime variable $\bar r_i=1$ ($\bar r_i=0$). Then (\ref{cpeqn2})-(\ref{cpeqn3}) follows from Theorem \ref{thmdetection}. 

The demand function $\hat q_i^{max}(\cdot)$ for incoming links  and the supply function $\bar q_i^{max}(\cdot)$ for outgoing links  are given by (\ref{thm1eqn7}): for $i=1,\,2$, if $\hat r_i(t)=1$, then $\hat q_i^{max}(t)=C_i$; otherwise if $\hat r_i(t)=0$, then according to (\ref{caseii}),  
$$\hat q_i^{max}(t)=q_i(t,\,b_i-)=q_i\left(t-{L_i\over k_i},\,a_i\right)=\bar q_i\left(t-{L_i\over k_i}\right)
$$ 
This shows (\ref{cpeqn5}). One can similarly show (\ref{cpeqn4}) using (\ref{thm1eqn7}) and (\ref{casei}). 

For $i=1,\,2$, if $u_i(t)=0$ which means the light is red, then the flow allowed through is zero, otherwise, it is given by(\ref{thm1eqn3}). This proves (\ref{cpeqn6}). 

(\ref{cpeqn7}) follows from the definition of the splitting parameters $\alpha_{i,k}$, $i=1,\,2,\,k=3,\,4$. (\ref{cpeqn7}) guarantees that at each time, there are one and only one incoming road  that has green light. 
\end{proof}

\subsection{Discrete-time formulation}\label{secdiscretetime}

In this section, we present the discrete-time version of the optimization problem in Theorem  \ref{cpthm}. Let us introduce a few more notations for the convenience of our presentation. Consider a uniform time grid
$$
0~=~t^0~<~t^1~\ldots~<~t^N~=~T,\qquad t^j-t^{j-1}~=~\delta t,\quad j~=~1,\ldots, N
$$
Throughout the rest of this article, we use superscript `j' to denote the discrete value evaluated at time step $t^j$.  In addition, we let  $L_i/k_i= \Delta^f_i\delta t,\,L_i/w_i=\Delta^b_i\delta t,\,\,\Delta_i^f\in\mathbb{N}, \, \Delta_i^b\in\mathbb{N},\,i=1, 2, 3, 4$.

Approximating the numerical integration with rectangular quadratures,  we write equality (\ref{cpeqn2}) and (\ref{cpeqn3})  in discrete time as

\begin{equation}\label{dtcpeqn1}
\begin{cases}
\displaystyle \delta t\sum_{j=0}^{k-\Delta_i^b} \hat q_i^j -\delta t\sum_{j=0}^{k}\bar q_i^j+\rho^{jam}L_i~\leq~\mathcal{M}\,(1-\bar r_i^k)\\
\displaystyle \delta t\sum_{j=0}^{k-\Delta_i^b} \hat q_i^j -\delta t\sum_{j=0}^{k}\bar q_i^j+\rho^{jam}L_i~\geq~-\mathcal{M}\,\bar r_i^k+\varepsilon
\end{cases}
\quad \Delta_i^b~\leq~k~\leq~N,\quad i~=~1,\,2,\,3,\,4
\end{equation}
\begin{equation}\label{dtcpeqn2}
\begin{cases}
\displaystyle \delta t\sum_{j=0}^{k-\Delta_i^f} \bar q_i^j -\delta t\sum_{j=0}^{k}\hat q_i^j~\leq~\mathcal{M}\,\hat r_i^k\\
\displaystyle \delta t\sum_{j=0}^{k-\Delta_i^f} \bar q_i^j -\delta t\sum_{j=0}^{k}\hat q_i^j~\geq~\mathcal{M}\,(\hat r_i^k-1)+\varepsilon
\end{cases}
\quad \Delta_i^f~\leq~k~\leq~N,\quad i~=~1,\,2,\,3,\,4
\end{equation}
where $\bar r_i^k,\,\hat r_i^k\in\{0,\,1\}$. $\mathcal{M}\in\mathbb{R}_+$ is a sufficiently large number, $\varepsilon\in\mathbb{R}_+$ is a sufficiently small number. Constraints (\ref{dtcpeqn1}) and (\ref{dtcpeqn2}) determines the regime variables associated with the two boundaries of each link. Once the flow phases are determined, the demand and supply functions (\ref{cpeqn4}), (\ref{cpeqn5}) are re-written in discrete time as
\begin{equation}\label{dtcpeqn3}
\begin{cases}
\displaystyle C_i-\mathcal{M}\,\bar r_i^j~\leq~\bar  q_i^{max,j}~\leq~C_i\\
\displaystyle  \hat  q_i^{j-\Delta^b_i}-\mathcal{M}\,(1-\bar r_i^j)~\leq~\bar q_i^{max,j}~\leq~ \hat  q_i^{j-\Delta^b_i}+\mathcal{M}\,(1-\bar r_i^j)
\end{cases}\qquad i~=~1,\,2,\,3,\,4
\end{equation}
\begin{equation}\label{dtcpeqn4}
\begin{cases}
\displaystyle C_i+\mathcal{M}\,(\hat r_i^j-1)~\leq~\hat q_i^{max,j}~\leq~C_i\\
\displaystyle \bar  q_i^{j-\Delta^f_i}-\mathcal{M}\,\hat r_i^j~\leq~\hat  q_i^{max,j}~\leq~\bar  q_i^{j-\Delta_i^f}+\mathcal{M}\,\hat r_i^j
\end{cases}\qquad\qquad\qquad~~ i~=~1,\,2,\,3,\,4
\end{equation}
Next, let us re-formulate (\ref{cpeqn6}). Introducing dummy variables $\zeta_1^j,\, \zeta_2^j$, $1\leq j\leq N$, such that 
\begin{equation}\label{zetadef}
\zeta_i^j~=~\min\left\{\hat q_i^{max, j},\, {\bar q_3^{max,j}\over \alpha_{i,3}}\, {\bar q_4^{max,j}\over \alpha_{i,4}}\right\}\qquad i~=~1,\,2
\end{equation}
Then the discrete-time version of (\ref{cpeqn6}) can be readily written as 
\begin{equation}\label{dtcpeqn5}
\begin{cases}
0~\leq~\hat q_i^j~\leq~\mathcal{M}\,u_i^j\\
\zeta_1^j+\mathcal{M}\,(u_i^j-1)~\leq~\hat q_i^j~\leq~\zeta_1^j
\end{cases}\qquad i~=~1,\,2,\quad j~=~1,\ldots, N
\end{equation}
In order to write (\ref{zetadef}) as linear constraints, one could  write it as three ``less or equal" statements, which is simple but bear the potential limitation of traffic holding. Instead, one may introduce additional binary variables $\xi_i^j,\, \eta_i^j$ and real variables $\beta_i^j$ for $i=1,\,2$, $j=1,\ldots, N$, such that (\ref{zetadef}) can be accurately formulated as
\begin{equation}\label{dtcpeqn6}
\begin{cases}
\bar q_3^{max, j}/ \alpha_{i,3}-\mathcal{M}\,\xi_i^j~\leq~\beta_i^j~\leq~\bar q_3^{max,j}/\alpha_{i,3}\\
\bar q_4^{max,j}/\alpha_{i,4}-\mathcal{M}\,(1-\xi_i^j)~\leq~\beta_i^j~\leq~\bar q_4^{max,j}/\alpha_{i,4}\\
\hat q_i^{max, j}-\mathcal{M}\,\eta_i^j~\leq~\zeta_i^j~\leq~\hat q_i^{max,j}\\
\beta_i^j-\mathcal{M}\,(1-\eta_i^j)~\leq~\zeta_i^j~\leq~\beta_i^j
\end{cases} \qquad i~=~1,\,2
\end{equation}
Finally, we have the obvious relations
\begin{equation}\label{dtcpeqn7}
\bar q_k^j(t)~=~\alpha_{1, k}\hat q_1^j(t)+\alpha_{2,k}\hat q_2^j(t)\qquad k~=~3,\,4,\qquad j~=~1,\ldots, N
\end{equation}
and 
\begin{equation}\label{dtcpeqn8}
u_1^j+u_2^j~=~1\qquad j~=~1,\ldots, N
\end{equation}

The proposed MILP formulation of signal control problem is summarized by (\ref{dtcpeqn1})-(\ref{dtcpeqn4}) and (\ref{dtcpeqn5})-(\ref{dtcpeqn8}). This formulation captures many desirable features of vehicular flow on networks such as physical queues, spill back, vehicle turning, and shock formation and propagation (although not explicitly). The signal control allows time-varying cycle length and splits, as well as the utilization of real-time information of traffic flows.

\subsection{Bound the separating shock}
At the end of this section, we discuss an additional linear constraint that ensures that the congested phase on each link is bounded. Such condition is closely related to travel delay:  if the congested phase remain bounded, there will be a reasonable upper bound for the travel time of each driver. Let us consider a single link $[a,\,b]$.  If we wish to bound the separating shock within the interval $[c,\,b]$ for some $a<c<b$,  this means (with the same notation as before) that
\begin{equation}\label{revision1}
\int_{0}^{\,t} q(s,\,c)\,ds ~\leq~\int_{0}^{\,t-{b-c\over w}}\hat q(s)\,ds+\rho^{jam}\,(b-c)
\end{equation}
(\ref{revision1}) follows by applying (\ref{latent3}) to the interval $[c,\,b]$. It is helpful to notice that since $[a,\,c]$ remains in the free flow phase, $q(s,\,c)$ must be equal to $\bar q\left(s-{c-a\over k}\right)$. Thus the condition for the congested phase to remain in $[c,\,b]$ becomes
\begin{equation}\label{boundedshock}
N_{up}\left(t-{c-a\over k}\right)~\leq~N_{down}\left(t-{b-c\over w}\right)+\rho^{jam}\,(b-c)
\end{equation}
Condition (\ref{boundedshock}) can be easily written as linear constraint in discrete time. 
If one wishes to include (\ref{boundedshock}) in the objective function instead of using it as a constraint, he/she may simply minimize the difference between the left and right hand sides of (\ref{boundedshock}).

\section{Numerical example}\label{secNE}

In this section, we consider the simple network consisting of two junctions, as shown in Figure \ref{fignetwork}. We will show the numerical result of optimal signal control at this junction obtained by the MILP summarized in the previous section.

\begin{figure}[h!]
\centering
\includegraphics[width=.45\textwidth]{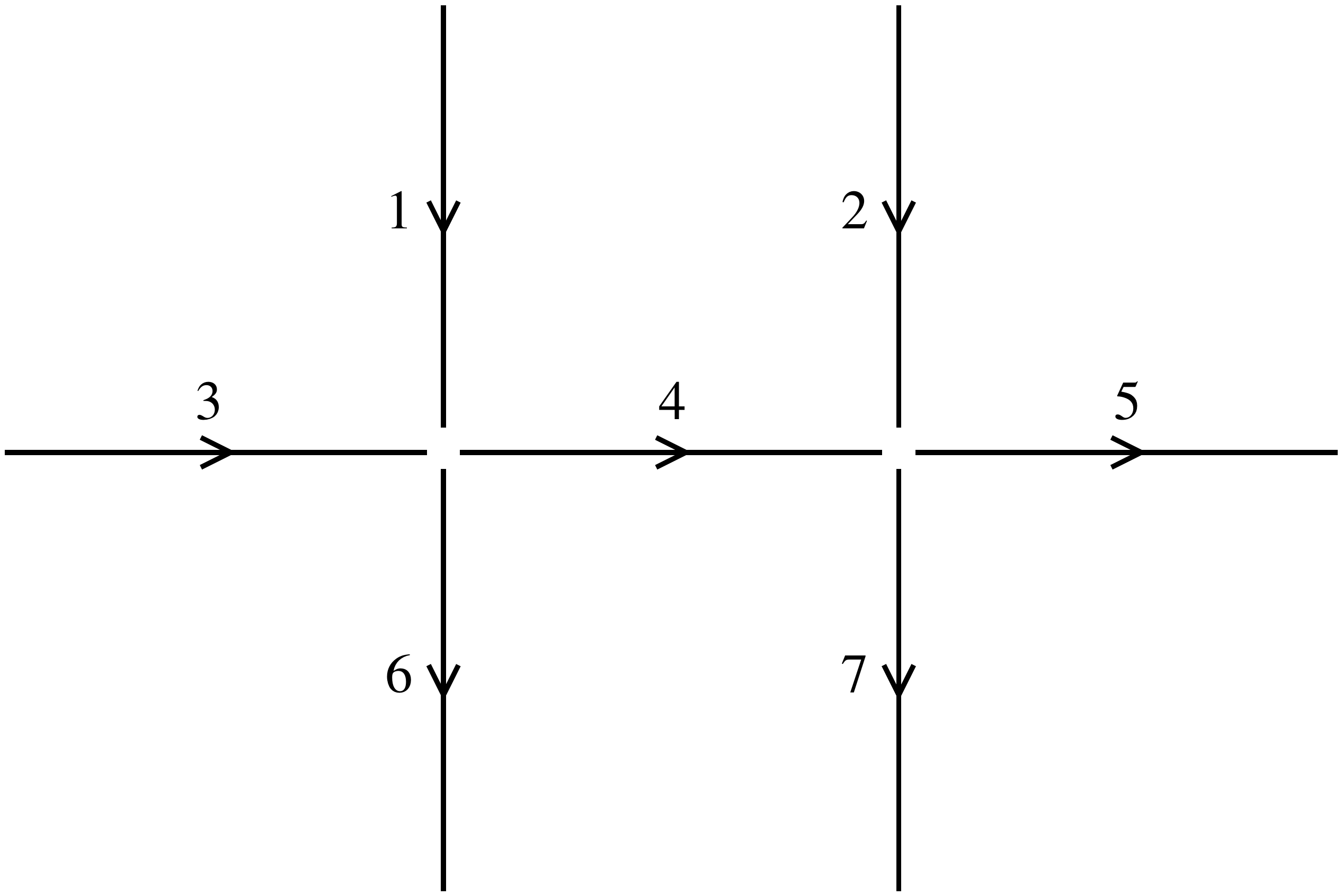}
\caption{Test network with two signalized intersections.}
\label{fignetwork}
\end{figure}

\subsection{Numerical setting}
We assume that the fundamental diagrams for all seven links are the same, that is, for $i=1,\ldots, 7$ 
$$
k_i=30 ~\hbox{mile/hour},\quad w_i=10~\hbox{mile/hour},\quad \rho_i^{jam}=400~\hbox{vehicle/mile},\quad C_i=3000~\hbox{vehicle/hour}
$$
The lengths of all links are set equally to be $0.3$ miles. We choose a time grid of 20 intervals and a time step of 0.005 hour (18 seconds). The flow entering links $I_1, I_2, I_3$ are chosen to be time-varying functions whose value at each time interval is randomly generated between $0$ veh/hour and 3000 veh/hour.  In addition, to ensure the performance of signal control, we include further constraint that the congested phase must not pass the mid-point of the link, i.e. it remains on the spatial interval $[0.15, 0.3]$. This is done by invoking constraint (\ref{boundedshock}). 

The MILP is  solved with ILOG Cplex 12.1.0, which runs with Intel Xeon X$5675$ Six-Core 3.06 GHz processor provided by the Penn State Research Computing and Cyberinfrastructure.

\subsection{Numerical results}

The solution time of the MILP described  above is 0.49 seconds.  In order to have a clear visualization of the optimal signal strategy and the separating shocks on each link, we use the boundary datum $\bar q_i,\,\hat q_i,\,i=1,\,2,\,3,\,4$ obtained from the optimal solution to construct solutions to the Hamilton-Jacobi equation (\ref{HJeqn}) using Lax-Hopf formula. For the H-J equation, the separating shock no longer represents discontinuity, rather, it is displayed as a `kink'  (discontinuity in the first derivative). The Moskowitz functions $N(t,\,x)$ for links $I_3,\,I_4$ are shown in Figure \ref{figNlink3} and \ref{figNlink4}, respectively. The Moskowitz functions for link $I_1$ and $I_2$ viewed from a different angle are shown in Figure \ref{figbirdeye1} and \ref{figbirdeye2}.

One can clearly observe, in each figure, two types of characteristics lines: forward ones and backward ones, representing the free flow and congested phase, respectively. The shared boundary of the two regimes is precisely the separating shock wave, which we managed to implicitly handle with variational method. It is also clear that the congested region never crosses the middle point of the link throughout the planning horizon, thanks to (\ref{boundedshock}).

\begin{figure}[h!]
\begin{minipage}[b]{.49\textwidth}
\centering
\includegraphics[width=\textwidth]{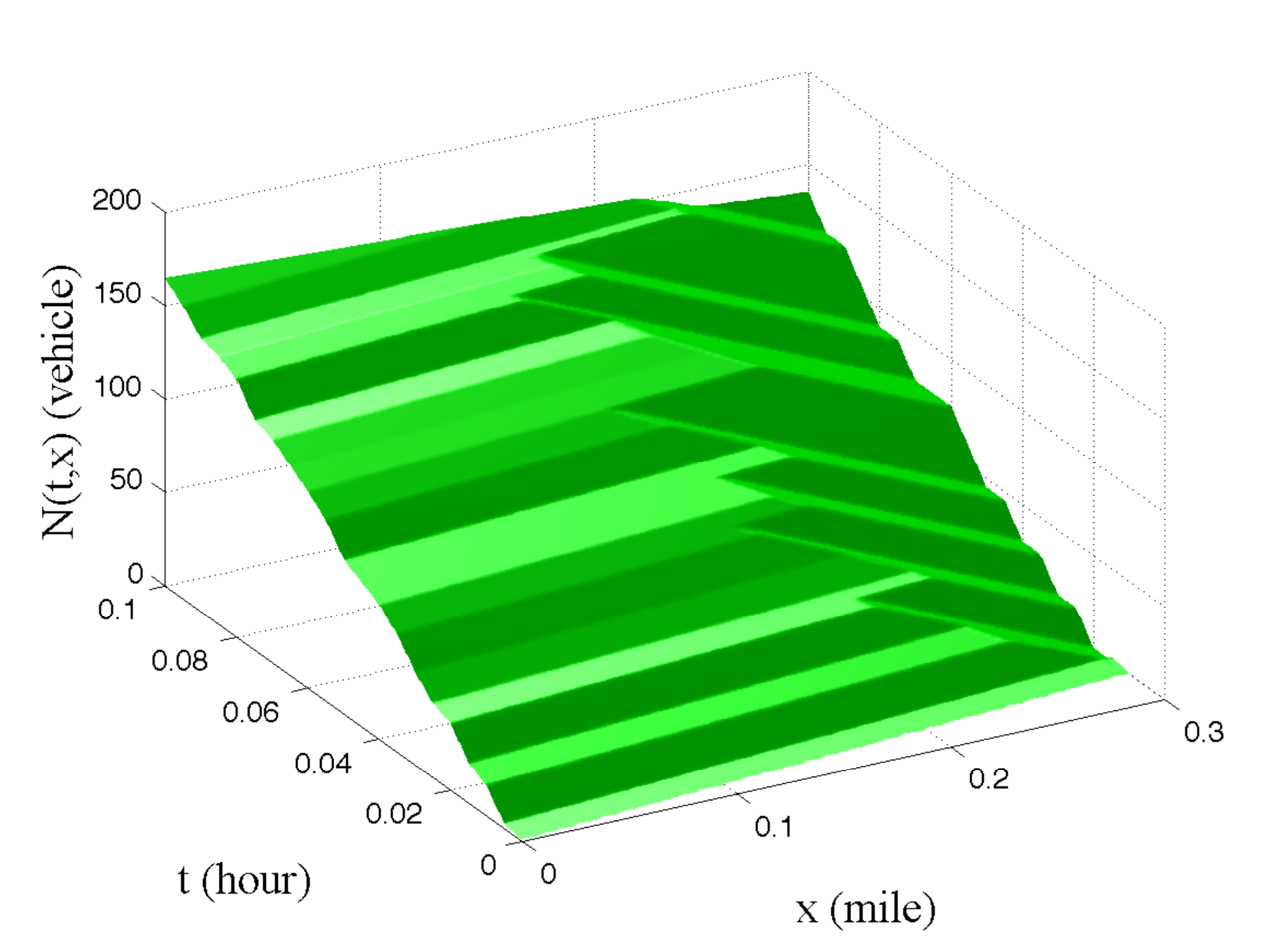}
\caption{Moskowitz function for link $I_3$.}
\label{figNlink3}
\end{minipage}
\begin{minipage}[b]{.49\textwidth}
\centering
\includegraphics[width=\textwidth]{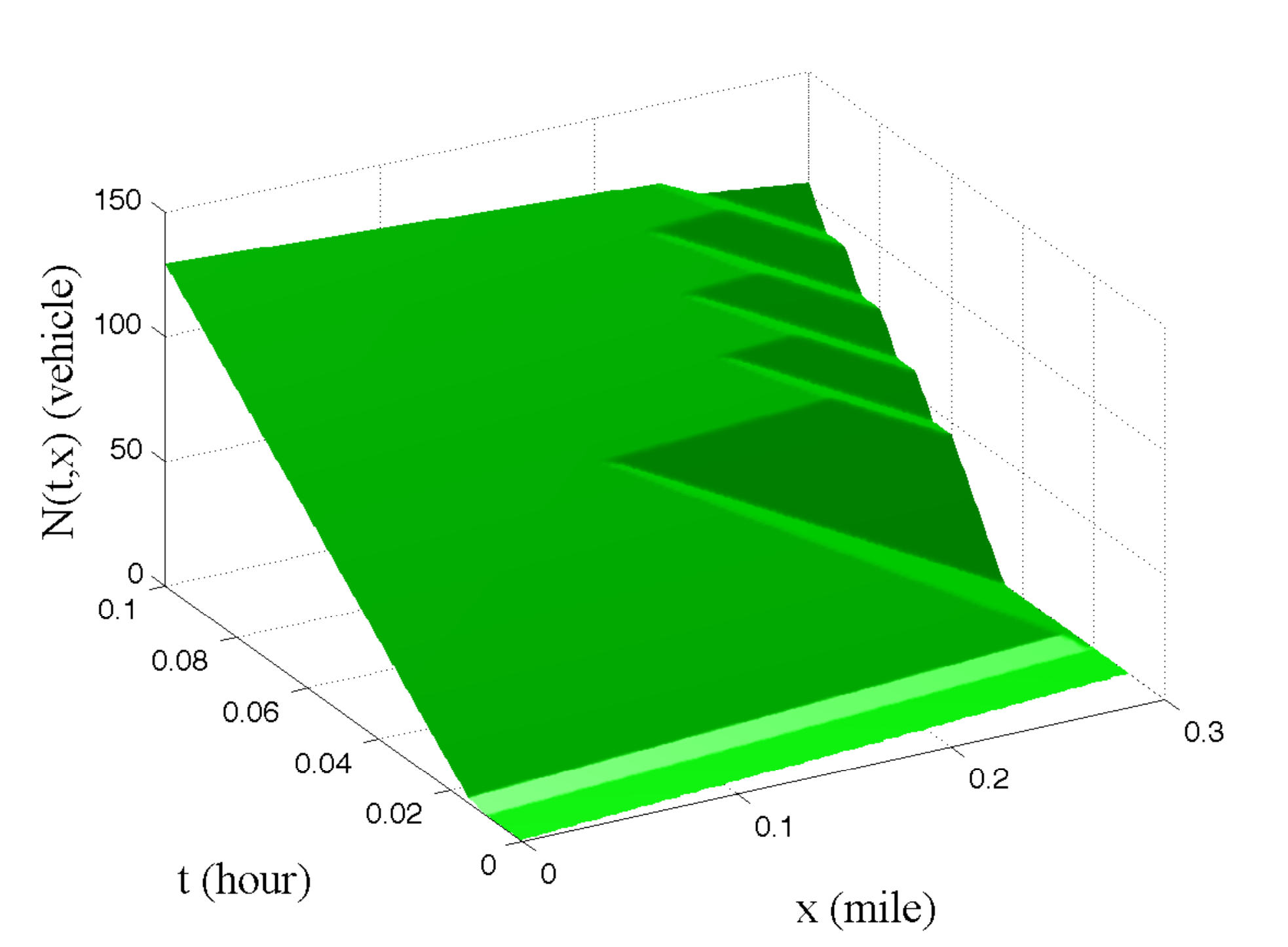}
\caption{Moskowitz function for link $I_4$.}
\label{figNlink4}
\end{minipage}
\end{figure}

\begin{figure}[h!]
\begin{minipage}[b]{.49\textwidth}
\centering
\includegraphics[width=\textwidth]{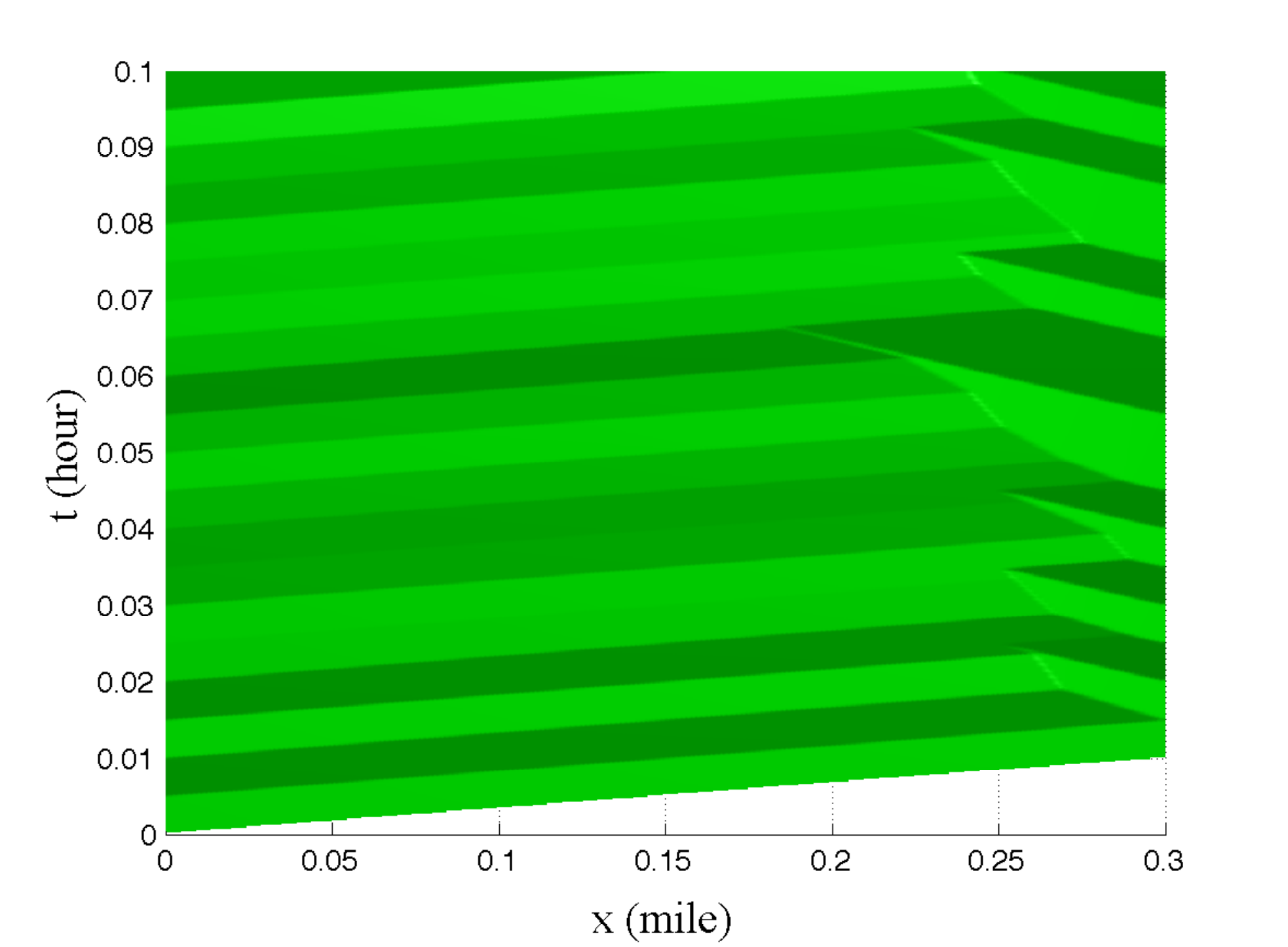}
\caption{Moskowitz function for link $I_1$.}
\label{figbirdeye1}
\end{minipage}
\begin{minipage}[b]{.49\textwidth}
\centering
\includegraphics[width=\textwidth]{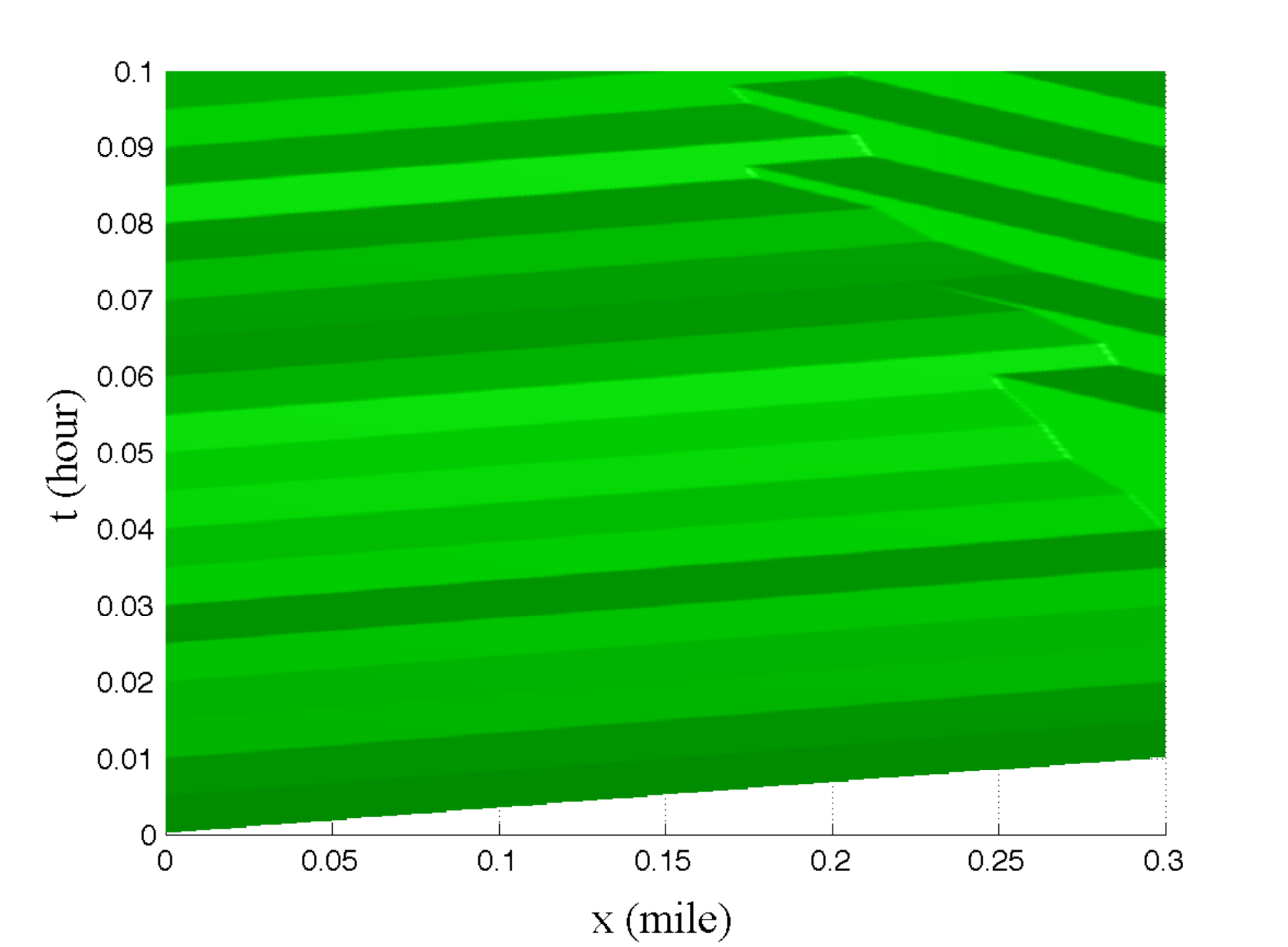}
\caption{Moskowitz function for link $I_2$.}
\label{figbirdeye2}
\end{minipage}
\end{figure}

%
%\begin{figure}[h!]
%\centering
%\includegraphics[width=.45\textwidth]{birdeye1.eps}
%\caption{Test network with two signalized intersections.}
%\label{figbirdeye1}
%\end{figure}

\section{Conclusion}
This paper proposes a link-based signal control problem. The key ingredient of our formulation is the use of the Lax-Hopf formula to detect latent shocks and hence to identity the regimes (free flow/congested) at the entrance and exit of each link. The analytical framework allows us to further bound the congested region and to avoid spillback. The problem of adaptive signal control is formulated in discrete time as a mixed integer linear program. The resulting MILP  requires fewer (binary) variables compared to cell-based approaches.

One limitation of the current approach is the lack of computational tractability, when the problem size scales up. Future research aims at exploring heuristic optimization algorithms such as the genetic algorithms; and new computational paradigms involving parallel computing.

\end{document}